\documentclass[11pt]{amsart}
\usepackage{blkarray}
\usepackage{floatrow}
\usepackage[usenames,dvipsnames]{xcolor} 
\usepackage{amsmath}
\usepackage[utf8]{inputenc}
\usepackage{amsfonts}
\usepackage{amssymb}
\usepackage{tikz-cd}
\usepackage{graphicx}
\usepackage{bm}
\usepackage[numbers,sort]{natbib}
\usepackage{enumitem}
\usepackage[margin=1.33in]{geometry}
\usepackage[pdftex,
pdfpagemode=UseNone,
breaklinks=true,
extension=pdf,
colorlinks=true,
linkcolor=blue,
citecolor=PineGreen,
urlcolor=blue,
]{hyperref}
\usepackage{amsthm}
\usepackage{comment}

\theoremstyle{theorem}
\newtheorem{theorem}{Theorem}[section]

\newtheorem{lemma}[theorem]{Lemma}
\newtheorem{corollary}[theorem]{Corollary}
\newtheorem{proposition}[theorem]{Proposition}

\theoremstyle{definition}

\newtheorem{definition}[theorem]{Definition}
\newtheorem{examplex}[theorem]{Example}
\newtheorem{notation}[theorem]{Notation}
\newenvironment{example}
{\pushQED{\qed}\examplex}
{\popQED\endexamplex}

\newtheorem{remarkx}[theorem]{Remark}
\newenvironment{remark}
{\pushQED{\qed}\remarkx}
{\popQED\endremarkx}

\DeclareMathOperator{\NN}{\mathbb{N}}

\DeclareMathOperator{\ZZ}{\mathbb{Z}}

\DeclareMathOperator{\Wr}{Wr}

\DeclareMathOperator{\Znn}{\mathbb{Z}_{\geqslant 0}}
\DeclareMathOperator{\arc}{Arc}
\newcommand{\bx}{\ensuremath{\mathbf{x}}}

\usepackage{csquotes}

\usepackage{listings}

\definecolor{dkgreen}{rgb}{0,0.6,0}
\definecolor{gray}{rgb}{0.5,0.5,0.5}
\definecolor{mauve}{rgb}{0.58,0,0.82}

\title{Wronskians form the inverse system\\ of the arcs of a double point}
\author{Rida Ait El Manssour}
\address{IRIF, CNRS, Université Paris Cité, Paris, France}
\email{manssour@irif.fr}

\author{Gleb Pogudin}
\address{LIX, CNRS, \'Ecole polytechnique, Institute Polytechnique de Paris, Paris, France}
\email{gleb.pogudin@polytechnique.edu}

\begin{document}

\begin{abstract}
    The ideal of the arc scheme of a double point or, equivalently, the differential ideal generated by the ideal of a double point is a primary ideal in an infinite-dimensional polynomial ring supported at the origin.
    This ideal has a rich combinatorial structure connecting it to  singularity theory, partition identities, representation theory, and differential algebra.

    Macaulay inverse system is a powerful tool for studying the structure of primary ideals which describes an ideal in terms of certain linear differential operators.
    In the present paper, we show that the inverse system of the ideal of the arc scheme of a double point is precisely a vector space spanned by all the Wronskians of the variables and their formal derivatives.
    We then apply this characterization to extend our recent result on Poincar\'e-type series for such ideals.
\end{abstract}

\maketitle
\section{Introduction}
The main object studied in this paper is the arc scheme of a double point. 
In algebraic geometry, for a $k$-variety $X$, the points of the arc space correspond to the Taylor coefficients of the $k[\![t]\!]$-points of $X$.
Thus, the arc space of $X$ can be viewed as an infinite-order generalization of the tangent bundle or, in other words, the space of formal trajectories on the variety~\cite{ArcLecture, arcsbook}.
For a variety defined by an ideal $I \subset k[\mathbf{x}]$, where $\mathbf{x} = (x_1, \ldots, x_n)$, the defining ideal $I^{\arc}$ of the arc space belongs to the ring of formal derivatives of $\mathbf{x}$
\[
  k[\mathbf{x}^{(\infty)}] := k[x_i^{(j)} \mid 1 \leqslant i \leqslant n,\; 0 \leqslant j].
\]
If we define $x_i(t) := \sum\limits_{j = 0}^\infty x^{(j)}t^j$, then $I^{\arc}$ is generated by the coefficients of $f(\mathbf{x}(t))$ with respect to $t$ for all $f \in I$.

We will be interested in the case when $X$ is a double point, that is, is defined by an ideal
\begin{equation}\label{eq:In}
  \mathcal{I}_n = \langle x_ix_j \mid 1 \leqslant i, j \leqslant n\rangle \subset k[x_1, \ldots, x_n].
\end{equation}
The ideal $\mathcal{I}_n^{\arc}$ defining its arc space corresponds only to one point over $k$, the origin, but exhibits a rich multiplicity structure which was initially studied in the context of differential algebra~\cite{x2,lowpower,Keefe1960,vertex}, appeared in representation theory~\cite{Feigin2002}, and recently attracted attention because of its connections (especially in the case $n = 1$) to partition identities~\cite{lotharing,Afsharijoo2021,afsharijoo2021andrewsgordon,Bruschek2012,AM2020,handbook,Bai2019}.

One classical approach to study such $\mathfrak{m}$-primary ideals of the maximal ideal $\mathfrak{m}$ at the origin is via considering linear differential operators with constant coefficients called \emph{Noetherian equations}. 
These operators map the given primary ideal to a subspace of $\mathfrak{m}$.
This approach goes back to Macaulay and Gr\"obner (see~\cite{Alonso2006} for a survey) and can be viewed as a far-reaching generalization of the characterization of the root multiplicities of univariate polynomials through the vanishing of the derivatives.
Noetherian equations provide an alternative representation for the corresponding ideals, and a generalization of this construction has been recently empolyed, for example, for numerical primary decomposition~\cite{Chen2022} and for using commutative algebra to solve linear PDEs~\cite{CidRuiz2021,AitElManssour2021}.

Since $\mathcal{I}_n^{\arc}$ is a primary ideal of the maximal ideal at the origin (although in the infinite-dimensional polynomial ring), a natural question is to find the Noetherian equations (or, equivalently, the Macaulay inverse system) for $\mathcal{I}_n^{\arc}$.
The main result of this paper is a concise and surprising answer to this question: 

\begin{displayquote}
\emph{The Macaulay inverse system of the arc scheme of a double point is spanned by the Wronskians of finite subsets of $\mathbf{x}^{(\infty)}$}.
\end{displayquote}
One intriguing aspect of the proof is that it borrows some ideas from a seemingly unrelated characterization of Null Lagrangians~\cite{Ball1981} from the calculus of variations.

In the finite-dimensional case, a natural invariant of a primary ideal contained within a maximal ideal is its dimension.
The dimension of $\mathcal{I}_n^{\arc}$ is infinite, but one consider instead the dimensions of the truncations $\dim k[\mathbf{x}^{(\leqslant h)}] / (\mathcal{I}_n^{\arc} \cap k[\mathbf{x}^{(\leqslant h)}])$, where $k[\mathbf{x}^{(\leqslant h)}] := k[x_i^{(j)} \mid 1 \leqslant i \leqslant n,\; 0 \leqslant j \leqslant h]$ (see~\cite{GlebRida} for details).
As an application, we use the computed Macaulay inverse system for $\mathcal{I}_n^{\arc}$ and the recently proved Kolchin-Schmidt conjecture~\cite{etesse2023schmidtkolchin} to compute these dimensions for the case of double points thus extending the main result of~\cite{GlebRida}.
Specifically, we show that
\begin{equation}\label{eq:dim}
\dim k[\mathbf{x}^{(\leqslant h)}] / (\mathcal{I}_n^{\arc} \cap k[\mathbf{x}^{(\leqslant h)}]) = (n + 1)^{h + 1}.
\end{equation}

The rest of the paper is structured as follows.
Section~\ref{sec:preliminaries} contains necessary preliminaries and notation. In Section~\ref{Sec:Main}, we state the main results of the paper. 
In Section~\ref{sec:Mult} we prove Theorem~\ref{thm:main}. Then we use it to characterize the inverse systems of the elimination ideals $\mathcal{I}_n^{\arc} \cap k[\mathbf{x}^{(\leqslant h)}]$ in Section~\ref{sec:Trunc}.
This characterization is then used in Section~\ref{sec:poincare} to establish the dimension result~\eqref{eq:dim}.

\textbf{Acknowledgements.}
We would like to thank Fulvio Gesmundo and Bernd Sturmfels for stimulation discussions and Bogdan Rai\c{t}\u{a} for pointing out the paper~\cite{Ball1981}.
 This work has partly been supported by the French ANR-22-CE48-0008 OCCAM project.

\section{Preliminaries}\label{sec:preliminaries}

\begin{notation}[Formal derivatives]
  Let $x$ be a variable in a polynomial ring.
  Then we will consider $x, x', x'', x^{(3)}, \ldots$ as independent polynomial variables, and, for any $h \in \Znn$, we introduce the following notation
  \[
    x^{(<h)} := (x, x', \ldots, x^{(h - 1)}) \quad \text{ and }\quad x^{(\infty)} := (x, x', x'', \ldots).
  \]
  $x^{(\leqslant h)}$ is defined analogously.
  Similarly, for a tuple $\mathbf{x} = (x_1, \ldots, x_n)$ of variables we have
  \[
    \mathbf{x}^{(< h)} := (x_1^{(< h)}, \ldots, x_n^{(<h)}) \quad\text{ and }\quad \mathbf{x}^{(\infty)} := (x_1^{(\infty)}, \ldots, x_n^{(\infty)}).
  \]
\end{notation}

\begin{definition}[Arc space~\cite{ArcLecture}]
    Let $I \subset k[\mathbf{x}]$ (where $\mathbf{x} = (x_1, \ldots, x_n)$) be an ideal defining a variety $X$.
    Then the defining ideal $I^{\arc}$ of the \emph{arc space} of $X$ belongs to the polynomial ring $k[\mathbf{x}^{(\infty)}]$ in formal derivatives of $\mathbf{x}$.
    If we denote $\mathbf{x}(t) := \sum\limits_{j = 0}^\infty \mathbf{x}^{(j)}t^j$, then $I^{\arc}$ is generated by the coefficients of $f(\mathbf{x}(t))$ with respect to $t$ for all $f \in I$.
\end{definition}

For defining inverse systems, we will restrict ourselves to ideals with the support at the origin in order to keep things simple.

\begin{definition}[Inverse system~\cite{Alonso2006}]
    Let $p, q \in k[\mathbf{x}]$ be polynomials.
    We define 
    \[
    p \bullet q = p\left(\frac{\partial}{\partial x_1}, \ldots, \frac{\partial}{\partial x_n}\right) q(\mathbf{x}).
    \]
    Let $I \subset k[\mathbf{x}]$ be an $\mathfrak{m}$-primary ideal, where $\mathfrak{m}$ is the maximal ideal generated by $\mathbf{x}$.
    Then the \emph{inverse system} $I^{\perp} \subset k[\mathbf{x}]$ is a vector space defined by
    \[
      I^{\perp} := \{f \in k[\mathbf{x}] \mid \forall p \in I \colon (f\bullet p) \in \mathfrak{m}\}.
    \]
\end{definition}

\begin{example}
    Let $I = \langle x^2, y^2, z - x- y \rangle \subseteq k [x,y,z],$ which is $\langle x,y,z \rangle$-primary. Then $I^\perp$ is the $k$-vector space generated by $\{ 1, x - y, x + z, xy + yz + xz + z^2\}.$
\end{example}

The following proposition gives a dual definition of the inverse system of a $\mathfrak{m}-$primary ideal as a solution space to linear differential equations with constant coefficients.
   \begin{proposition}\label{pro:dualdefinition}
    For every $\mathfrak{m}$-primary ideal $I$ of $k[\bx]$ where $\mathfrak{m} = \langle x_1, \ldots , x_n \rangle $, the following characterization of  $I^\perp$ holds:
  \[
        I^\perp = \left\{ P \in k[\bx]\, | \, \forall\, f \in I, f \bullet P =0 \right\}.  
    \]
\end{proposition}
\begin{proof}
For $\bm{\alpha} = (\alpha_1, \ldots, \alpha_n) \in \ZZ^n$ we abbreviate $\bx^{\bm{\alpha}} =x_1^{\alpha_1} \cdots x_n^{\alpha_n} $.
    Note that, for every two monomials $\bx^{\bm{\alpha}}$ and $\bx^{\bm{\beta}}$ we have: 
    \[  
        \bx^{\bm{\beta}} \bullet \bx^{\bm{\alpha}} = 
            \begin{cases}
              \tfrac{\alpha !}{ (\alpha - \beta)!}  \bx^{\bm{\alpha} - \bm{\beta}} \, \text{ if } \bm{\beta} \leqslant \bm{\alpha} \\
                0
            \end{cases}
    \]
    where $\bm{\alpha}! := \alpha_1 !  \cdots \alpha_n!$.
    Let $P = a_{u_0} \bx^{u_0} + \ldots + a_{u_k} \bx^{u_k} $  and $f = \sum_{u \in \NN^n} f_u \bx^u  \in I$, such that $f \bullet P =0$. This implies that,
    \[ 
        f \bullet P = \sum_{v \leqslant u_0} a_{u_0} f_{v} \tfrac{u_0!}{(u_0 - v)!} \bx^{ u_0 - v }+ \ldots + \sum_{v \leq u_k} a_{u_k} f_{v} \tfrac{u_k!}{(u_k - v)!} \bx^{ u_k - v } = 0.
    \]
    Thus, the coefficient of $\tfrac{1}{u!} \bx^u$ in the last equation is  $a_{u_0} f_{u_0 - u} u_0! + \cdots + a_{u_k} f_{ u_k - u} u_k!$ which is equal to the remainder of $P \bullet ( f \cdot \bx^u)$ modulo $\mathfrak{m}$. Therefore $f \bullet  P =0 $ if and only if, $P \bullet ( f \cdot \bx^u) \in  \mathfrak{m}$ for  all monomial $\bx^u$. i.e., $P \bullet f \in \mathfrak{m} $ for all $f \in   I $ if and only if $f \bullet P =0$ for all $f \in   I $.
\end{proof}

\begin{notation}[Wronskians]
We define a derivation on $k[\mathbf{x}^{(\infty)}]$ by $(x_i^{(j)})' := x_i^{(j + 1)}$ for every $1 \leqslant i \leqslant n$ and $j \geqslant 0$.
Then, for polynomials $f_1, \ldots, f_\ell \in k[\mathbf{x}^{(\infty)}]$, we define the \emph{Wronskian  determinant} $\Wr(f_1, \ldots, f_\ell) $ as follows:
\begin{align}\label{eq: HanWron}
\Wr (f_1, \ldots, f_\ell) := \begin{vmatrix}
f_1 &  f_2 & \ldots & f_\ell \\
f_1' & f_2' & \ldots & f_\ell'\\
\vdots & \vdots & \vdots & \vdots\\
f_1^{(\ell - 1)} & f_2^{(\ell - 1)} & \ldots & f_\ell^{(\ell - 1)}\\
\end{vmatrix}.
\end{align}
\end{notation}

\section{Main results}\label{Sec:Main}

Our first main result characterizes the inverse system of the arc space of a double point defined by the ideal  $\mathcal{I}_n$~\eqref{eq:In}.

\begin{theorem}\label{thm:main}
    For every integer $n > 0$, the inverse system $(\mathcal{I}_n^{\arc})^\perp$ of the arc space of a double point is spanned by the Wronskians $\Wr(S)$, where $S$ ranges over all finite subsets~$S \subset \mathbf{x}^{(\infty)}$.
\end{theorem}

\begin{example}
    Let us determine the elements of $(\mathcal{I}_1^{\arc})^\perp \subseteq k[x^{(\infty)}]$ of order and degree at most $2$. 
    The  generators of $\mathcal{I}_1^{\arc}$ involving variables of order at most $2$ are $ G = \{ x^2, xx', 2xx'' + (x')^2 ,  x x^{(3)} + x' x'', 2x x^{(4)} + 2x' x^{(3)} +  (x'')^2  \}$. Consider the following polynomial operator 
    \[
    P = c_0 + c_1 \frac{\partial}{\partial x} + c_2 \frac{\partial}{\partial x'} + c_3 \frac{\partial}{\partial x''} + c_4 \frac{\partial^2}{\partial x^2} + c_5 \frac{\partial^2}{\partial (x')^2} + c_6 \frac{\partial^2}{\partial (x'')^2} + c_7 \frac{\partial^2}{\partial x x'} + c_8 \frac{\partial^2}{\partial x x''} + c_9 \frac{\partial^2}{\partial x' x''}
    \]
    by applying $P$ to the polynomials in $G$ we realize that the resulting polynomials belongs to $\mathfrak{m}$ if and only if 
    \[
    c_4 = c_6 = c_7 = c_9 = 0 \text{ and } c_5 + c_8 = 0. 
    \]
    Therefore the vector space $(\mathcal{I}_1^{\arc})^{\perp}\cap k[x,x',x'']_{ \leq 2}$ is generated by $\{ 1, x, x', x'', xx'' - (x')^2\}$. In fact, one could even show that it is also equal to  $(\mathcal{I}_1^{\arc})^{\perp}\cap k[x,x',x'']$.
\end{example}

We use Theorem~\ref{thm:main} to derive the following Poincare-type series for $\mathcal{I}_n^{\arc}$ extending~\cite[Theorem~3.1]{GlebRida}.

\begin{theorem}\label{thm:dim}
    For every positive integer $n$, we have
    \[
    \sum\limits_{h = 0}^\infty \dim \left( k[\mathbf{x}^{(\leqslant h)}] / (\mathcal{I}_n^{\arc} \cap k[\mathbf{x}^{(\leqslant h)}]) \right) t^h = \frac{n + 1}{1 - (n + 1)t}.
    \]
\end{theorem}

\section{Proof of Theorem~\ref{thm:main}}\label{sec:Mult}

We introduce new transcendental constants $\xi, \alpha_1, \ldots, \alpha_n$ (and denote $\bm{\alpha} = (\alpha_1, \ldots, \alpha_n)$), and consider the following differential operator with coefficients in $k[\xi , \bm{\alpha}]$ 
\[ 
D_{\xi, \bm{\alpha}} := \sum\limits_{j = 1}^n \sum\limits_{i = 0}^\infty \alpha_j \xi ^i \frac{\partial}{\partial x_j^{(i)}}.
\]
The key property of this operator is given by the lemma below:
\begin{lemma}\label{lem:D2_property}
    Let $P \in k [ \bf{x}^{(\infty)} ]$, then $P \in (\mathcal{I}_n^{\arc})^\perp$ if and only if $D_{\xi, \bm{\alpha}}^2(P) = 0$.
\end{lemma}

\begin{proof}
    We will consider $Q = D_{\xi, \bm{\alpha}}^2(P)$ as a polynomial in $\xi, \bm{\alpha}$. 
    Since $D_{\xi, \bm{\alpha}}$ is linear in $\bm{\alpha}$, the monomials that may appear in $Q$ will be of the form $\alpha_i\alpha_j\xi^\ell$ for $1 \leqslant i \leqslant j \leqslant n$ and $\ell \geqslant 0$.
    Then direct computation shows that the coefficient in front of this monomial is
    \begin{equation}\label{eq:expandD}
    C\sum\limits_{s = 0}^\ell \frac{\partial^2}{\partial x_i^{(s)}\partial x_j^{(\ell - s)}}(P),
    \end{equation}
    where $C = 1$ if $i = j$ and is equal to $2$ otherwise.
    This expression can be rewritten as $C (\sum\limits_{s = 0}^\ell x_i^{(s)}x_j^{(\ell - s)}) \bullet P$, and the polynomial $\sum\limits_{s = 0}^\ell x_i^{(s)}x_j^{(\ell - s)}$ is the coefficient at $t^\ell$ in $x_i(t) x_j(t)$.
    Therefore, vanishing of~\eqref{eq:expandD} for all $1 \leqslant i \leqslant j \leqslant n$ and $\ell \geqslant 0$ is equivalent to $P \in (\mathcal{I}_n^{\arc})^\perp$, and this finishes the proof.
\end{proof}

Lemma~\ref{lem:D2_property} implies that the elements of $(\mathcal{I}_n^{\arc})^\perp$ satisfy the following ``linearity under exponential perturbations'' property.

\begin{lemma}\label{lem:LinearityMulti}
Let $P \in k [ \bf{x}^{(\infty)} ]$, then $P \in (\mathcal{I}_n^{\arc})^\perp$ if and only if the following equality holds in the ring $k[\mathbf{x}^{(\infty)},  \bm{\alpha}][\![ \xi, t]\!]$ with the derivation extended by $t' = 1$:
\[ 
    P(\mathbf{x} + \bm{\alpha} e^{\xi t}) = P(\mathbf{x}) + e^{\xi t} Q(\xi, \bm{\alpha}, \bx).
\]
where $Q \in k[\xi, \bm{\alpha}, \bx^{(\infty)}]$.
Furthermore, if such equality holds, then $Q = D_{\xi, \bm{\alpha}}(P)(\mathbf{x})$.
\end{lemma}

\begin{proof}
First we observe that, for every $P \in k[\bx^{(\infty)}]$, we have
\[
 \frac{\partial}{\partial (e^{\xi t})} P(\bx + \bm{\alpha}e^{\xi t}) = D_{\xi, \bm{\alpha}} P(\bx + \bm{\alpha}e^{\xi t}).
\]
Therefore, we can write the Taylor expansion of $P(\bx + \bm{\alpha}e^{\xi t})$ with respect to $e^{\xi t}$ using the operator $D_{\xi, \bm{\alpha}}$:
\[
P(\bx + \bm{\alpha} e^{\xi t}) = P(\bx) + e^{\xi t} D_{\xi, \bm{\alpha}}P(\bx) + \frac{e^{2\xi t}}{2} D_{\xi, \bm{\alpha}}^2P(\bx) + \ldots.
\]
The statement of the lemma follows from the expansion above and Lemma~\ref{lem:D2_property}.
\end{proof}

\begin{corollary}\label{cor:vanish}
    Let $P \in (\mathcal{I}_n^{\arc})^\perp$ be homogeneous of degree $d + 1$.
    Then, for any $\xi_1, \ldots, \xi_d \in k$ and $\bm{\alpha_1}, \ldots, \bm{\alpha_d} \in k^n$, we have
    \[
    P(\bm{\alpha_1}e^{\xi_1 t} + \ldots + \bm{\alpha_d}e^{\xi_d t}) = 0.
    \]
\end{corollary}

\begin{proof}
    We consider $\xi_1, \ldots, \xi_d, \bm{\alpha_1}, \ldots, \bm{\alpha_d}$ as unknowns.
    Let $\bx^\ast = \bm{\alpha_1}e^{\xi_1 t} + \ldots + \bm{\alpha_{d - 1}}e^{\xi_{d - 1} t}$
    By Lemma~\ref{lem:LinearityMulti} we have  
    \[
        P(\bx^\ast + \bm{\alpha_d}e^{\xi_d t}) = P(\bx^\ast) + e^{\xi_d t} D_{\xi_d, \bm{\alpha_d}}P(\bx^\ast).
    \]   
    This implies that $e^{\xi_{d} t }$ appears in $P(\bm{\alpha_1}e^{\xi_1 t} + \ldots + \bm{\alpha_d}e^{\xi_d t})$ with degree at most $1$. 
    Therefore, for every $1 \leqslant j \leqslant d$, we have   $e^{\xi_j t }$ appearing with degree at most $1$. 
    However, $P$ is homogeneous of degree $d + 1$, so it mush vanish on $\bm{\alpha_1}e^{\xi_1 t} + \ldots + \bm{\alpha_d}e^{\xi_d t}$.
\end{proof}

From this, we will deduce that every element of $(\mathcal{I}_n^{\arc})^\perp$ of degree $d$ vanishes on solutions of low order linear differential equations.

\begin{proposition}\label{prop:vanish_on_subspace}
    Let $V \subset k[\![t]\!]$ be a subspace of dimension $d$ closed under differentiation.
    Then every homogeneous $P \in (\mathcal{I}_n^{\arc})^\perp$ of degree $d + 1$ vanishes on any point of $V^n$.
\end{proposition}

\begin{proof}
    We fix a space $V$ and a polynomial $P$ as in the statement of the proposition.
    By the cyclic vector theorem~\cite{CHURCHILL2002}, there exists $f \in k[\![t]\!]$ satisfying a linear differential equation of order $d$ such that $V = \langle f, f', \ldots, f^{(d - 1)} \rangle$.
    Let $\mu_1, \ldots, \mu_\ell$ and $e_1, \ldots, e_\ell$ be the roots and their multiplicities of the characteristic polynomial of the minimal linear equation for $f$.
    Then the minimality implies that $\sum e_i = d$.
    
    If $e_1 = \ldots = e_\ell = 1$, then $\ell = d$ and every element of $V$ is a linear combination $e^{\mu_1 t}, \ldots, e^{\mu_\ell t}$.
    By Corollary~\ref{cor:vanish}, $P$ vanishes on any $n$-tuple of linear combinations of these exponentials.

    Now we consider the case when at least one of $e_i$'s is greater than one.
    In this case, any element of $V$ is a quasi-polynomial of the form $\sum\limits_{i = 1}^\ell \sum\limits_{j = 0}^{e_i - 1}a_{i, j} t^j e^{\mu_i t}$.
    We will first consider the case $k = \mathbb{C}$, and
    we claim that any such function can be represented as a point-wise limit of linear combinations of at most $d$ exponentials.
    Indeed, using the Taylor expansion we can write $e^{(\mu_i + \frac{r}{N})t} $ for every $r < j$ as:
    \[
    e^{(\mu_i + \frac{r}{N}t)} = e^{\mu_i t} + \frac{r}{N} te^{\mu_i} + \ldots + \frac{r^j}{j! N^j} t^j e^{\mu_i t} + o\left( \frac{1}{N^j}\right).
    \]
   These equations for $0 \leqslant r \leqslant j$ allow us to express $t^j e^{\mu_i t}$ as follows:
    \[
    t^je^{\mu_i t} = \lim_{N \rightarrow \infty}
    \begin{pmatrix}
        1 & 0 & \ldots & 0
    \end{pmatrix}
    \begin{pmatrix}
        \frac{1}{j!}\left(\frac{j}{N}\right)^j & \frac{1}{(j - 1)!}\left(\frac{j}{N}\right)^{(j - 1)} & \ldots & 1\\
        \frac{1}{j!}\left(\frac{j - 1}{N}\right)^j & \frac{1}{(j - 1)!}\left(\frac{j - 1}{N}\right)^{(j - 1)} & \ldots & 1\\
        \vdots & \ddots & \ddots & \vdots \\
        0 & 0 & \ldots & 1
    \end{pmatrix}^{-1} 
    \begin{pmatrix}
        e^{(\mu_i + \frac{j}{N})t}\\
        e^{(\mu_i + \frac{j - 1}{N})t}\\
        \vdots\\
        e^{\mu_i t}
    \end{pmatrix}.
    \]
    Therefore, the whole linear combination $\sum\limits_{i = 1}^\ell \sum\limits_{j = 0}^{e_i - 1}a_{i, j} t^j e^{\mu_i t}$ can be written as a limit $N \to \infty$ of linear combinations of $\{e^{(\mu_i + j/N) t} \mid 1 \leqslant i \leqslant \ell, \; 0 \leqslant j < e_i\}$. We take a tuple $\bx^\ast \in V^n$ and approximate it by such a sequence of linear combinations of exponentials $\bx^\ast_N \to \bx^\ast$.
    Since all the considered functions are smooth, the limit is uniform on bounded domains and, therefore, commutes with derivation.
    Thus $P(\bx^\ast_N) \to P(\bx^\ast)$, so the latter is equal to zero.

    We finish by extending the proof for the case when at least one of $e_i$ is greater than one to an arbitrary ground field $k$ of characteristic zero.
    We observe that the property $P \in (\mathcal{I}_n^{\arc})^\perp$ is defined by a system of linear equations over $\mathbb{Q}$ on the coefficients of $P$.
    Therefore, we can express $P$ as a linear combination of elements of $(\mathcal{I}_n^{\arc})^\perp$ with the coefficients in $\mathbb{Q}$, so we will further assume that $P$ has coefficients in $\mathbb{Q}$. 
    If we consider $\mu_i$'s and $a_{i, j}$'s as variables and plug expressions of the form $\sum\limits_{i = 1}^\ell \sum\limits_{j = 0}^{e_i - 1}a_{i, j} t^j e^{\mu_i t}$ to $P(\bx)$, the result will be a combinations of different products of the exponentials with the coefficients being polynomials in $\mu_i$'s and $a_{i, j}$'s. 
    These polynomials vanish over $\mathbb{C}$, so they are identically zero and, thus, vanish over any field $k$.
\end{proof}

\begin{definition}
    For any positive integers $n$ and $h$,  we define the infinite Hankel matrix as follows:
\begin{align}\label{eq : HankelMulti}
\mathbf{H}_{n, h}:= \begin{pmatrix}
x_1 & x_2 & \cdots & x_n & x_1' & x_2' &  \cdots & \cdots\\
x_1' & x_2'& \cdots & x_n' &  x_1'' &  x_2'' & \cdots & \cdots\\
x_1'' & x_2'' & \cdots & x_n'' & x_1^{(3)} & x_2^{(3)} & \cdots & \cdots\\
\vdots & \vdots & \vdots& \vdots  & \vdots & \vdots & \vdots & \cdots\\
x_1^{(h - 1)} & x_2^{(h -1)} & \cdots & x_n^{(h - 1)}& x_1 ^{(h)} & x_2^{(h)} & \cdots & \cdots
\end{pmatrix}
\end{align}
Since $\mathbf{H}_{n, h}$ is a submatrix of $\mathbf{H}_{n, h + 1}$, an infinite-by-infinite matrix $\mathbf{H}_{n, \infty}$ can be defined by taking $h \to \infty$.
\end{definition}

\begin{proposition}\label{pro: HanWron2 }
Every minor of the infinite Hankel matrix~\eqref{eq : HankelMulti} belongs to $(\mathcal{I}_n^{\arc})^\perp$.
\end{proposition}

\begin{proof}
Consider one such $\ell \times \ell$-minor.
It can be written as $\det (x_{s_j}^{(a_i + b_j)})_{i, j}$, where $(a_1, \ldots, a_\ell)$, $(b_1, \ldots, b_\ell)$, and $(s_1, \ldots, s_{\ell})$ are tuples of integers.
We denote this determinant by $W(\bx)$. 
Since $(\bx + \bm{\alpha} e^{\xi t})^{(k)} = \bx^{(k)} + \xi^k \bm{\alpha} e^{t \xi}$ we can write $W( \bx + \bm{\alpha} e^{t \xi})$ using the linearity of the determinant as follows:
\[
 W( \bx + \bm{\alpha} e^{t \xi}) = W(\bx) + e^{t \xi} S_1 + e^{2t \xi} S_2 + \ldots,
\]  
 where $S_1$ is a sum of determinants given by replacing a column 
 \[ ^\mathsf{t}\begin{pmatrix}
 x_{s_j}^{(b_j + a_1)}, \ \ 
 x_{s_j}^{(b_j + a_2)}, \ \ 
 \cdots, & 
 x_{s_j}^{(b_j + a_\ell)}
 \end{pmatrix}
 \]
 of $\Wr(\bx)$ by 
 \[^\mathsf{t}\begin{pmatrix}
 \alpha_{s_j}\xi^{b_j + a_1},&
 \alpha_{s_j}\xi^{b_j + a_2},&
 \cdots,&
  \alpha_{s_j}\xi^{b_j + a_\ell}
 \end{pmatrix},\]
 and $S_2$ is a sum of determinants given by performing two such replacements in $W(\bx)$, and so on for all $S_k$ where $k \geqslant 1$.
 Since  two columns  of the form  \[^\mathsf{t}\begin{pmatrix}
 \alpha_{s_j}\xi^{b_j + a_1},&
 \alpha_{s_j}\xi^{b_j + a_2},&
 \cdots,&
  \alpha_{s_j}\xi^{b_j + a_\ell}
 \end{pmatrix},\,  
 ^\mathsf{t}\begin{pmatrix}
 \alpha_{s_k}\xi^{b_k + a_1},&
 \alpha_{s_k}\xi^{b_k + a_2},&
 \cdots,&
  \alpha_{s_k}\xi^{b_k + a_\ell}
 \end{pmatrix},\] are proportional, we have $S_k = 0$ for $k > 1$ and, thus
 \[
 W(\bx + e^{\xi t}) = W(\bx) + e^{\xi t} S_1. 
 \]
Hence by Lemma~\ref{lem:LinearityMulti} we conclude that $W(\bx) \in (\mathcal{I}_n^{\arc})^{\perp}$.
\end{proof}

\begin{lemma}\label{lemma:minors}
Consider the infinite Hankel matrix $
\mathbf{H}_{n, h}$ and let $\mathbf{J}_{n, h}$ the ideal generated by the $h \times h$ minors of $\mathbf{H}_{n, h}$. Then $\mathbf{J}_{n, h}$ is a prime  differential ideal.
\end{lemma}

The proof of Lemma \ref{lemma:minors} uses the concept of $1$--generic matrices. For more details see \cite{eisenbud, guerrieri, Eisenbud2}). 

\begin{definition}
We say that a matrix $M$ with the entries being linear forms in some variables is $1$--generic if for every nonzero vectors $v_1$ and $v_2$ over the ground field:
\[ 
^\mathsf{t}v_1 M v_2 \neq 0.
\]
\end{definition}

\begin{remark}
Hankel matrices are examples of $1$--generic matrices~\cite[p. 549]{eisenbud}. 
Moreover the maximal minors of a $1$--generic matrix form a prime ideal \cite[Theorem~1]{Eisenbud2}.
\end{remark}

\begin{proof}[Proof of Lemma \ref{lemma:minors}]
    For every $k\ge h - 1$  the following matrix:
    \tiny
\[
\setcounter{MaxMatrixCols}{20}
\mathbf{H}_{n, h,k}:= \begin{pmatrix}
x_1 & x_2 & \cdots & x_n & x_1' & x_2' &  \cdots & x_n' & \cdots  & x_1^{(k)} & x_2^{(k)} & \cdots & x_n^{(k)}\\
x_1' & x_2'& \cdots & x_n' &  x_1'' &  x_2'' & \cdots & x_n'' & \cdots  & x_1^{(k + 1)} & x_2^{(k + 1)} & \cdots & x_n^{(k + 1)}\\
x_1'' & x_2'' & \cdots & x_n'' & x_1^{(3)} & x_2^{(3)} & \cdots & x_n^{(3)} & \cdots  & x_1^{(k + 2)} & x_2^{(k + 2)} & \cdots & x_n^{(k + 2)}\\
\vdots & \vdots & \vdots& \vdots & \vdots & \vdots & \vdots & \vdots & \vdots  & \vdots & \vdots & \vdots & \vdots\\
 x_1^{(h - 1)} & x_2^{(h - 1)} & \cdots & x_n^{(h - 1)} & x_1^{(h )} & x_2^{(h)} &  \cdots & x_n^{(h)} & \cdots  & x_1^{(h - 1 + k)} & x_2^{( h - 1 + k)} & \cdots & x_n^{( h - 1 + k)}
\end{pmatrix}
\]
\normalsize
is $1$--generic since $\mathbf{H}_{n, h,k}$ is block matrix of Hankel matrices (up to exchanging columns). Therefore the maximal minors of $\mathbf{H}_{n, h,k}$ form a prime ideal $\mathbf{J}_{n, h,k}$. 
Hence the ideal $\mathbf{J}_{n, h} = \bigcup_k \mathbf{J}_{n, h,k}$ is a prime ideal.
Furthermore, since a derivative of every maximal minor of $\mathbf{H}_{n, h}$ is a linear combination of maximal minors of the same matrix, this ideal is a differential ideal.
\end{proof}

\begin{theorem}\label{thm: SolvI2KABIRA}
Let $P \in k[\bf{x}^{(\infty)} ]$ such that $P \in  (\mathcal{I}_n^{\arc})^\perp$.  Then $P$ is a linear combination  of Wronskians $\Wr(S)$, where $S$ ranges over all finite subsets~$S \subset \mathbf{x}^{(\infty)}$.
\end{theorem}

\begin{proof}
Let $P \in (\mathcal{I}_n^{\arc})^\perp $, without loss of generality we can suppose that $P$ is homogeneous of degree $d$. Let $\bx(t) := (x_1(t), \ldots, x_n(t))$ be an $n$-tuple of power series from $\overline{k}[\![t]\!]$ which belongs to $V(\mathbf{J}_{n, d})$.
The fact that $\bx(t)$ satisfies every equation in $V(\mathbf{J}_{n, d})$ implies that the dimension of the vector space spanned by $x_1(t), \ldots, x_n(t)$ and their derivatives is less than $d$.
Then, by Proposition~\ref{prop:vanish_on_subspace}, $P$ vanishes at $\bx(t)$.
Therefore, by the Hilbert's Nullstellensatz, $P$ belongs to $ \mathbf{J}_{n, d}$.
Since the generators of $ \mathbf{J}_{n, d}$ are homogeneous of degree $d$ as well as $P$, $P$ must be a linear combination of the generators, that is, of the maximal minors of $\mathbf{H}_{n, d}$. 
\end{proof}

Proposition~\ref{pro: HanWron2 } and Theorem~\ref{thm: SolvI2KABIRA} imply:

\begin{corollary}\label{cor:inf_by_inf}
The vector space $(\mathcal{I}_n^{\arc})^\perp$ is equal to the vector space spanned by the minors of $\mathbf{H}_{n, \infty}$.
\end{corollary}

\section{Inverse system and elimination}\label{sec:Trunc}

\begin{definition}
    Let $V$ be a finite dimensional subspace of $k[\bx]$ (with $\bx = (x_1, \ldots, x_n)$), that is stable under derivative with respect to any variable $x_1, \ldots, x_n$. We denote by $\mathfrak{I}(V)$ the ideal defined as follows:
    \[
        \mathfrak{I}(V) : = \left\{ f \in k[\bx]\, : \, \forall \, P \in V, f \bullet P = 0  \right\}.
    \] 
\end{definition}

\begin{proposition}
     Let $V$ be a finite dimensional subspace of $k[\bx]$, that is stable under derivative with respect to any variable $x_1, \ldots, x_n$, then
     \[
     \mathfrak{I}(V) =\left\{ f \in k[\bx]\, : \,\forall \, P \in V,  P \bullet f \in \mathfrak{m}  \right\}
     \]
\end{proposition}
\begin{proof}
Let $P := \sum_{u \in \NN^n} a_u \bx^u \in V$, and let $f = f_{u_0} \bx^{u_0} + \cdots + f_{u_k} \bx^{u_k} \in k[\bx]$. Then 
\begin{equation}\label{eq:f bullet P}
    f \bullet P = \sum_{u_0 \leq u} a_u f_{u_0}\tfrac{u!}{(u -u_0)!}\bx^{u - u_0 } + \cdots + \sum_{u_k \leq u} a_u f_{u_k}\tfrac{u!}{(u - u_k)!}\bx^{u - u_k }.  
\end{equation}
Note that, the coefficient of $\tfrac{1}{\beta!} \bx^\beta$ in \eqref{eq:f bullet P} is 
$$a_{ u_0 + \beta} f_{u_0} (u_0 + \beta )! + \cdots + a_{ u_k + \beta} f_{u_k} (u_k + \beta )! .$$
Moreover $\bx^\beta \bullet P = \sum_{\beta \leq u} a_u \tfrac{u!}{(u - \beta)!} \bx^{u - \beta} \in V$, and 
\[
    ((\bx^\beta \bullet P) \bullet f) (0) = a_{u_0 + \beta }f_{u_0} (u_0 + \beta) ! + \cdots + a_{ u_k + \beta} f_{u_k} (u_k + \beta )!.
\]
Hence $f \bullet P = 0$ if and only if $ ((\bx^\beta \bullet P) \bullet f) (0) = 0$ for all $\beta$. 
Therefore $f \bullet P = 0$ for all $P \in V$ if and only if, $ (P \bullet f) (0) = 0$ for all $P \in V$.
\end{proof}

\begin{lemma}\label{lemma:Elim}
     Let $J \subset k[\bx]$ be a $\mathfrak{m}$-primary ideal. Then for every $m < n$ we have:
     \[
       (J \cap k[x_1,\ldots, x_m])^\perp = \left\{ P|_{\{ x_s = 0, \, \forall \, m < s \leq n \}}\, : \, P \in J^\perp\right\}.
    \]
\end{lemma}
\begin{proof}
See \cite[Proposition 3.2]{Marinari-Moller-Mora}
\end{proof}

\begin{lemma}\label{lemma:Elim2}
     Let $J \subset k[\{x_i\, |\, i \in \NN\}]$ an $\mathfrak{m}$-primary homogeneous ideal where $\mathfrak{m}:= \langle \{x_i\, | \, i \in \NN \} \rangle$. Then for every $\ell \in \NN$ we have:
     \[
        (J \cap k[x_0, \ldots, x_\ell])^\perp = \left\{ P|_{\{ x_s = 0, \, \forall \, s>\ell  \}}\, : \, P \in J^\perp\right\}.
     \]
\end{lemma}
\begin{proof}
 One inclusion follows immediately  from the fact that if $\forall f \in J, f \bullet P = 0$, then $\forall  f\in J \cap k[x_0, \ldots, x_\ell ], f \bullet P|_{\{ x_s = 0, \, \forall \, s> \ell   \}} =0 $ . Let us prove  the opposite inclusion:
 \[
        (J \cap k[x_0, \ldots, x_\ell])^\perp\subseteq \left\{ P|_{\{ x_s = 0, \, \forall \, s> \ell   \}}\, : \, P \in J^\perp\right\}.
 \]
 
 Let $P \in (J \cap k[x_0, \ldots, x_\ell])^\perp$, then by Lemma \ref{lemma:Elim} there exists  $P_1 \in (J \cap k[x_0, \ldots, x_{\ell + 1}])^\perp$ such that $P_1|_{x_{\ell +1 = 0}} = P $. By repeating this  one can construct a power series $\Tilde{P} \in k[\![\{x_i\, |\, i \in \NN\}]\!]$ such that $\Tilde{P}|_{\{ x^{(s)} = 0, \, \forall \, s > \ell \}} = P$, and
    \[
        f \bullet \Tilde{P} = 0 \text{ for all } f \in J.
    \]
    For every $d \in \NN$, let $\Tilde{P}_d$  be the homogeneous polynomial summand of $\Tilde{P}$ of degree $d$. Since $J$ is a homogeneous polynomial ideal with respect to the degree  we have:
    \[
   \forall \, d  \in \NN, \, \forall\, f \in J:\,    f \bullet \Tilde{P}_{d} = 0 .
    \]
    Hence, $\forall \, d \in \NN,\, \Tilde{P}_{d} \in J^\perp$. 
    
    Let $m\in \NN$ such that,  $\forall\,  0 \leq i \leq n : x_i^m \in J$, then $\deg_{x_i}(\Tilde{P}) < m$.
    Let us  introduce
    \[
        Q = \sum_{ d \leq (m - 1)(\ell + 1)} \Tilde{P}_d.
    \]
    We would like to prove that $Q|_{\{ x_s = 0, \, \forall \, m < s \}} = P$.  Since the monomial $x_0^{m -1} x_1^{m -1 } \cdots x_\ell^{m - 1} $ has degree $d_m := ( m - 1)(\ell +1)$, every monomial summand of $\Tilde{P}$ whose degree  bigger than $d_m$  is involving $x_s$ for some  $s > \ell$. Therefore $\Tilde{P}|_{\{ x_s= 0, \, \forall \, s> \ell \}} = Q|_{\{ x_s = 0, \, \forall \, s> \ell \}} = P$. Hence, 
    \[
    P \in \left\{ P|_{\{ x_s = 0, \, \forall \, s>\ell \}}\, : \, P \in J^\perp\right\}.
    \]
\end{proof}

\begin{theorem}\label{thm : Elim}
For every non negative integer $h$, we have
\[
    (\mathcal{I}_n^{\arc}\cap k[\bx^{\leqslant h}])^\perp = \left\{ P|_{\{ x^{(s)} = 0, \, \forall \, s > h \}}\, : \, P \in ( \mathcal{I}_n^{\arc})^\perp\right\}.
\]
\end{theorem}
\begin{proof}
Since $(\mathcal{I}_n^{\arc}\cap k[\bx^{\leqslant h}])$ is homogeneous with respect to the degree, this follows immediately  from Lemma \ref{lemma:Elim2}
\end{proof}

We can make this statement more explicit using Corollary~\ref{cor:inf_by_inf}.

\begin{notation}
    For a differential variable $x$ and positive integer $h$, we denote by $\mathbf{T}_h(x)$ the upper-triangular matrix of order $h + 1$ with $x^{(h - i)}$ on the $i$-th diagonal above the main one.
    That is:
    \[
    \mathbf{T}_h(x) = \begin{pmatrix}
        x^{(h)} & x^{(h - 1)} & x^{(h - 2)} & \ldots & x \\
        0 & x^{(h)} & x^{(h - 1)} & \ldots & x' \\
        \vdots & \vdots & x^{(h)} & \ddots & \vdots \\
        \vdots & \vdots & \vdots & \ddots & x^{(h - 1)}\\
        0 & 0 & 0 & \ldots & x^{(h)}
    \end{pmatrix}.
    \]
    We define $\mathbf{T}_{n, h}$ to be the $(h + 1) \times n(h + 1)$ matrix obtained by concatenation of $\mathbf{T}_{h}(x_1), \ldots, \mathbf{T}_{h}(x_n)$.
\end{notation}

\begin{corollary}\label{cor:main_at0}
The vector space    $(\mathcal{I}_n^{\arc}\cap k[\bx^{\leqslant h}])^\perp$ is equal to the vector space generated by the minors of $\mathbf{T}_{n, h}$.
\end{corollary}


\section{Poincar\'e-type series for $\mathcal{I}_n^{\arc}$}\label{sec:poincare}

The goal of this section is to use the obtained characterization of $(\mathcal{I}_n^{\arc}\cap k[x^{\leqslant h}])^\perp$ in order to prove the following dimension result (cf.~\cite[Theorem~3.1]{GlebRida}).

\begin{theorem}\label{thm:poincare}
    For every $h \in \mathbb{Z}_{\geqslant 0}$ we have
    $\dim k[\mathbf{x}^{(\leqslant h)}] / (\mathcal{I}_n^{\arc} \cap k[\mathbf{x}^{(\leqslant h)}]) = (n + 1)^{h + 1}$.
\end{theorem}

\begin{notation}
    For a differential variable $x$ and positive integer $h$, we denote by $\mathbf{S}_h(x)$ the upper-triangular matrix of order $h + 1$ with $x^{(i + 1)} / i!$ on the $i$-th diagonal above the main one.
    That is:
    \[
    \mathbf{S}_h(x) = \begin{pmatrix}
        x & x' & \frac{x''}{2} & \ldots & \frac{x^{(h)}}{ h!} \\
        0 & x & x' & \ldots & \frac{x^{(h - 1)} }{ (h - 1)!} \\
        \vdots & \vdots & \ddots &  \ddots & \vdots \\
        \vdots & \vdots & \vdots & \ddots & x'\\
        0 & 0 & 0 & \ldots & x
    \end{pmatrix}.
    \]
    We define $\mathbf{S}_{n, h}$ to be the $(h + 1) \times n(h + 1)$ matrix obtained by concatenation of the matrices $\mathbf{S}_{h}(x_1), \ldots, \mathbf{S}_{h}(x_n)$.
\end{notation}

\begin{lemma}\label{lem:T_to_S}
    For every $n, h$, the dimensions of the spaces of the minors of $\mathbf{T}_{n, h}$ and $\mathbf{S}_{n, h}$ coincide.
\end{lemma}

\begin{proof}
    The bijective map between these spaces is defined by $x^{(i)} \to x^{(h - i)} / (h - i)!$.
\end{proof}

\begin{lemma}\label{lem:S_to_S}
    For every $n, h$, the space of the minors of $\mathbf{S}_{n, h}$ coincides with the space of maximal minors of $\mathbf{S}_{n + 1, h}|_{x_{n + 1} = 1}$.
\end{lemma}

\begin{proof}
    We observe that $\mathbf{S}_h(x_{n + 1})|_{x_{n + 1} = 1}$ is simply the $(h + 1)$-dimensional identity matrix.
    Therefore, every minor of $\mathbf{S}_{n, h}$ can be  completed to a maximal minor of $\mathbf{S}_{n + 1, h}|_{x_{n + 1} = 1}$ by taking the columns with the ones on the missing rows.
    In the other direction, every maximal minor of $\mathbf{S}_{n + 1, h}|_{x_{n + 1} = 1}$ can be reduced to a minor of $\mathbf{S}_{n, h}$ by discarding the columns in the $\mathbf{S}_h(x_{n + 1})$-part and the rows in which these columns contain ones.
\end{proof}

\begin{lemma}\label{lem:diff_homog}
    The dimension of the space of maximal minors of $\mathbf{S}_{n + 1, h}|_{x_{n + 1} = 1}$ is equal to $(n + 1)^{h + 1}$.
\end{lemma}

\begin{definition}
    A differential polynomial $P\in k[\mathbf{x}^{(\infty)}]$ of degree $d$ is called \emph{differentially homogeneous} if, for a differential indeterminate $y$, the equality $P(y\cdot \mathbf{x}) = y^d P(\mathbf{x})$ holds.
\end{definition}

\begin{proof}[Proof of Lemma~\ref{lem:diff_homog}]
    By~\cite[Theorem~2.5.3 and Proposition~2.1.1]{etesse2023schmidtkolchin} the space of maximal minors of $\mathbf{S}_{n + 1, h}$ is exactly the space of differentially homogeneous polynomials of degree $h + 1$, and its dimension is equal to $(n + 1)^{h + 1}$.
    Furthermore, by~\cite[Proposition~1.3]{Reinhart1996}, the restriction of the map $P \to P|_{x_{n + 1} = 1}$ on the space of differentially homogeneous polynomials of fixed degree is injective.
    Therefore, the dimension of the space of maximal minors of $\mathbf{S}_{n + 1, h}|_{x_{n + 1} = 1}$ is also equal to $(n + 1)^{h + 1}$.
\end{proof}

\begin{proof}[Proof of Theorem~\ref{thm:poincare}]
    The concatenation of Lemmas~\ref{lem:T_to_S}, \ref{lem:S_to_S}, and~\ref{lem:diff_homog} implies that the dimension of the space of minors of $\mathbf{T}_{n, h}$ is equal to $(n + 1)^{h + 1}$.
    Then Corollary~\ref{cor:main_at0} implies that $\dim (\mathcal{I}_n^{\arc}\cap k[x^{\leq h}])^\perp = (n + 1)^{h + 1}$. Moreover, 
    \cite[Proposition~5.5]{Marinari-Moller-Mora} implies that 
    \[
    \dim k[x^{\leq h}] / (\mathcal{I}_n^{\arc}\cap k[x^{\leq h}]) = (n + 1)^{h + 1}.\qedhere
    \]
\end{proof}


\bibliographystyle{abbrvnat}
\bibliography{bibdata}
\end{document}